\documentclass[11pt,a4paper]{amsart}
\usepackage[latin1]{inputenc}
\usepackage{amsmath}
\usepackage{amsfonts}
\usepackage{amssymb}
\usepackage{hyperref}
\usepackage{graphicx}
\usepackage{subfig}
\usepackage{float}
\usepackage{booktabs}
\usepackage{caption}
\usepackage{subdepth}
\usepackage{color}
\usepackage{mathrsfs}
\usepackage[symbol]{footmisc}
\usepackage[top=4cm, bottom=4cm, left=3.5cm, right=3.5cm]{geometry}

\newcommand{\Rbb}{\mathbb{R}}

\newcommand{\Hbb}{\mathbb{H}}
\newcommand{\Vbb}{\mathbb{V}}
\newcommand{\Vcal}{\mathcal{V}}
\newcommand{\Ucal}{\mathcal{U}}
\newcommand{\Wcal}{\mathcal{W}}
\newcommand{\rmd}{\,{\rm d}}

\newtheorem{remark}{Remark}[section]
\newtheorem{lemma}{Lemma}[section]
\newtheorem{theorem}{Theorem}[section]
\newtheorem{proposition}{Proposition}[section]
\newtheorem{corollary}{Corollary}[section]
\numberwithin{equation}{section}

\newcommand{\overbar}[1]{\mkern 1.0mu\overline{\mkern-1.0mu#1\mkern-1.0mu}\mkern 1.0mu}

\newcommand{\vertiii}[1]{{\vert\kern-0.25ex\vert\kern-0.25ex\vert #1 
    \vert\kern-0.25ex\vert\kern-0.25ex\vert}}

\begin{document}
\graphicspath{{images/}}

\sloppy 

\pagestyle{myheadings}
\thispagestyle{plain}

\title{Electrode modelling: The effect of contact impedance}

\author{J\'er\'emi Dard\'e and Stratos Staboulis}


\maketitle

\begin{abstract}
The most realistic model for current-to-voltage measurements of electrical impedance tomography is the {\em complete electrode model} which takes into account electrode shapes and contact impedances at the electrode/object interfaces. When contact impedances are small, numerical instability can be avoided by replacing the complete model with the {\em shunt model} in which perfect contacts, that is zero contact impedances, are assumed. In the present work we show that using the shunt model causes only a (almost) linear error with respect to the contact impedances in modelling {\em absolute} current-to-voltage measurements. Moreover, we note that the electric potentials predicted by the two models exhibit genuinely different Sobolev regularity properties. This, in particular, causes different convergence rates for finite element approximation of the potentials. The theoretical results are backed up by two dimensional numerical experiments.
%
\end{abstract}

%

\section{Introduction}\label{sec:intro}

The modelling of current-to-voltage measurements is a fundamental part of {\em Electrical Impedance Tomography (EIT)} in which the aim is to reconstruct information about the conductivity distribution inside a body by external
measurements of electric current and voltage \cite{Barber84,Borcea02,Cheney99,Uhlmann09}. In practice, through a set of surface electrodes, currents of prescribed magnitudes are conducted into the object and the voltages needed for maintaining the currents are recorded. The most accurate mathematical model for a practical current-to-voltage measurement is the {\em complete electrode model} (CEM) which takes into account both the shapes and shunting effect of the electrodes by modelling the electrodes as medium-sized perfect conductors. Moreover, the quality of electrode contacts is described in the CEM by contact impedance parameters which model the effect of the resistive layers present at the electrode/object interfaces. It has been experimentally verified that the CEM is capable of predicting experimental data to better than $.1\%$ \cite{Cheng89,Somersalo92}. 

In {\em absolute} EIT, where conductivity images are computed from fixed-frequency current-to-voltage data measured on an unchanging object, a major challenge is that the measurement is significantly affected by unknown contact impedances. The problem can be tackled e.g.~by applying CEM-based iterative (Newton-type) methods which allow estimating both the conductivity distribution and the contact impedances \cite{Harhanen,Vauhkonen99,Vilhunen02a}. In this technique, a subtlety arises if a physical contact impedance is very close to zero, as numerical approximation of the CEM is known to turn unstable in the limit \cite{Stenberg08}. It has been observed that most numerical instabilities can be avoided by forcing the contact impedance parameters far enough from zero during the iterations \cite{Darde13}. Another tempting approach is to replace the CEM by the so-called {\em shunt model} (SM) which formally corresponds to the CEM with zero contact impedances. To rigorously justify either of the stabilization procedures, it is necessary to demonstrate that the CEM converges controllably to the SM as the contact impedance tends to zero. Proving and analyzing the convergence is the main research problem of this work.


To put our work into perspective, we note that the relationship between the CEM and other widely used idealistic models for EIT measurements --- the {\em continuum model} (CM) and the {\em point electrode podel} (PEM) --- has already been studied \cite{Hanke11,Hyvonen09}. The CM assumes infinite dimensional boundary data and has been used to prove the theoretical stability and uniqueness results on the inverse conductivity problem of EIT. The approximation of CM boundary data in terms of the CEM as the number of electrodes increase was considered in \cite{Hyvonen09}. In the PEM, the electrodes are modelled as point sources and the model has some attractive properties (such as conformal invariance and closed form solutions \cite{Hakula11}). However, due to its inbuilt singularity, the PEM is applicable only in modelling the difference of two electrode measurements. The interpretation of the PEM as a limiting case of the CEM as the electrodes get small was analyzed in \cite{Hanke11}. In conclusion, the work in this paper can be seen as complementing the above results from the point of view of the SM. 

Let us briefly outline the mathematics in this work. In a given object $\Omega\subset \Rbb^n$ with conductivity distribution $\sigma$, applying a static electric field in $\Omega$ induces a potential determined via the conductivity equation 
\begin{equation}\label{eq:condeq}
{\rm div}(\sigma\nabla u) = 0 \qquad {\rm in} \ \ \Omega.
\end{equation}
The effect of contact impedance (resistance) is modelled in the CEM by a positive $z$ in a Robin-type boundary condition
\begin{equation}\label{eq:Rob}
u + z\nu\cdot\sigma\nabla u = U
\end{equation}
on each electrode, with $U\in \Rbb$ representing the voltage perceived by the electrode. As current injection is exclusively confined to the electrodes, one ends up with a mixed zero-Neumann/Robin (NR) boundary value problem for \eqref{eq:condeq}. For fixed boundary data $U$ and $\sigma \equiv 1$, the asymptotics of the solution $u$ as $z$ tends to zero has been previously studied in \cite{Costabel96}; in particular, the limit coincides with the solution of the mixed zero-Neumann/Dirichlet (ND) problem. Unfortunately, the existing theory does not directly imply asymptotics for the solution of the CEM problem because the electrode voltages are not fixed --- they depend on $z$. Identifying the SM with a subspace projection of the CEM allows us to prove convergence in an abstract functional analytical framework. It is noteworthy that the obtained convergence result not only holds for the CEM and SM, but also for their discrete Galerkin approximations. 

To obtain an explicit convergence rate, we resort to the existing regularity theory for elliptic mixed boundary value problems \cite{Costabel96,Hyvonen09,Savare97}. In general, the solutions of NR and ND problems belong at best to $\smash{H^{1+s}(\Omega)}$ for any $s < 1$ (NR) and $\smash{s < \tfrac{1}{2}}$ (ND), respectively. We show that given this Sobolev regularity, the $\smash{H^1(\Omega)}$-error between the CEM and SM solutons is of order $O(z^s)$ with an arbitrary $s<\tfrac{1}{2}$\footnote{In \cite{Costabel96} the exponent $s = \tfrac{1}{2}$ is shown to be optimal for a generic two dimensional problem for Laplacian with fixed boundary data \eqref{eq:Rob}.}. Interestingly, the subspace projection property of the SM yields a better --- of order $O(z^{s})$ for any $s < 1$ --- rate for the electrode voltage. In other words, if the CEM is replaced by the SM, the error in the practical electrode measurement data exhibits almost linear dependence on the contact impedance. Finally, as a side-product, we point out that the same argument is also applicable to the finite element (FE) approximation of the CEM by piecewise linear polynomials: we show that the electrode potential $U$ is ``$\log$ two times'' more accurately approximated than the spatial potential $u$. 

The theoretical results are complemented with two numerical experiments in the plane. In the first one, we probe the convergence with respect to $z$. The anticipated convergence rates are detected with a reasonable success. In particular, the differing rates for $u$ and $U$ are observed.  
In the second numerical experiment, the convergence rates (with respect to the mesh parameter $h$) of the FE approximations of the CEM and SM are tracked. The results indicate that the corresponding rates are indeed different for $u$ and $U$. Moreover, the dissolution of the regularity of $u$ as $z$ tends to zero is shown to cause a slowing $h$-FE convergence rate for the CEM --- an analogous phenomenon is well-known in numerical analysis \cite{Costabel96b}. The regularizing effect of the contact impedance has practical significance as one usually aims for good electrode contacts whereas many reconstruction algorithms rely on repetitive accurate applications of FE-based solvers \cite{Darde13,Heikkinen02,Vauhkonen99}. 

The article is organized as follows. In Sec.~\ref{sec:model} we give the precise mathematical definition of the CEM together with the relevant notation. In the beginning of Sec.~\ref{sec:Gal} the SM is formulated as a subspace projection of CEM, and it is proven that the CEM potentials or their Galerkin approximations converge with an unspecified rate to the SM counterparts as the contact resistance tends to zero. The convergence rates between the models (in smooth geometry) are then derived in Sec.~\ref{sec:Con}. The obtained convergence rate as well as the effect of contact impedance to the convergence of a FE approximation are numerically studied in Sec.~\ref{sec:Num}. Finally, the concluding remarks are presented in Sec.~\ref{sec:Con}.

\section{Complete electrode model}\label{sec:model}

 Let $\Omega\subset\Rbb^n$, $n=2,3$ be a bounded domain (open and simply connected) with Lipschitz regular boundary $\partial\Omega$. The areas covered by electrodes are modelled by $M\geq 2$ mutually disjoint, well-separated subdomains $\{E_m\}_{m=1}^M$ on $ \partial\Omega $ and their union is abbreviated by $ E $.
 The conductivity $ \sigma\colon \Omega \to \Rbb^{n\times n} $ is assumed to be symmetric and such that there are constants $ \sigma_\pm > 0 $ satisfying
\begin{equation}\label{eq:sigma} \sigma_-|\xi|^2 \leq \xi^{\rm T}\sigma \xi \leq \sigma_+ |\xi|^2
\end{equation}
for all $ \xi\in\Rbb^n $ almost everywhere in $ \Omega $, that is to say, $\sigma$ is (possibly) anisotropic and somewhere between an ideal conductor and resistor. We assume that all electrodes are used for both current injection and voltage measurement, and we denote the amplitudes of the static net currents and voltage patterns by $I, U \in \Rbb^M$, respectively. According to the current conservation law, in the absence of sinks and sources we have the necessary condition
\[
I \in \Rbb_\diamond^M := \Bigg\{V\in \Rbb^M \ \colon \ \sum_{m=1}^M V_m = 0\Bigg\}.
\]
The contact impedances at the electrode/object interfaces are are modelled by positive real numbers 
\begin{equation}\label{eq:zzz}
0 < z_- \leq z_m \leq z_+, \qquad m=1,2,\ldots, M.
\end{equation}

\begin{remark}\label{rem:real-complex}
A time-harmonic current input would require taking the reactance into account, i.e., assuming $\sigma, z$ as complex valued such that the real parts satisfy \eqref{eq:sigma} and \eqref{eq:zzz}, respectively (see e.g.~\cite{Somersalo92}). With the suitable modifications, most of the results in this paper can be generalized to the complex case.
\end{remark}

The boundary value problem corresponding to the CEM is as follows: given an input current $ I\in\Rbb_\diamond^M $, find the induced potential pair 
\[ 
\Ucal := (u,U) \in H^1(\Omega)\oplus \Rbb_\diamond^M =: \mathbb{H}^1 
\] 
that satisfies weakly
\vspace{0.1cm}
\begin{equation}
\label{eq:mem}
\begin{array}{ll}
\displaystyle{{\rm div}(\sigma\nabla u) = 0 \qquad}  &{\rm in}\;\; \Omega, \\[6pt] 
{\displaystyle{\nu\cdot\sigma\nabla u} = 0 }\qquad &{\rm on}\;\;\partial\Omega\setminus{E},\\[6pt] 
{\displaystyle u+z_m{\nu\cdot\sigma\nabla u} = U_m } \qquad &{\rm on}\;\; E_m, \\[2pt] 
{\displaystyle \int_{E_m}\nu\cdot\sigma\nabla u \rmd S = I_m}, \qquad & m=1,2,\ldots M, \\[4pt]
\end{array}
\vspace{0.1cm}
\end{equation}
where $\nu:\partial\Omega \to \Rbb^n$ is the exterior unit normal of $\partial\Omega$. Note that since in practice only potential differences can be measured, we determine the ground level by fixing $U \in \Rbb_\diamond^M$. Furthermore, here we have implicitly assumed $\sigma$ to be smooth enough so that all the above objects have a meaning in the sense of traces. For a thorough physical justification of \eqref{eq:mem} the reader is advised to consult e.g.~\cite{Borcea02,Cheng89}.

Unique solvability of \eqref{eq:mem} in the Hilbert space $ \mathbb{H}^1 $ is a consequence of the Lax--Milgram lemma \cite{Dautray88}. The associated bilinear form $ B = B(\sigma,z) \colon \Hbb^1 \times \Hbb^1 \to \Rbb $ is defined by \cite{Somersalo92}
\begin{equation}\label{eq:sesq}
B(\Vcal,\Wcal) = \int_\Omega \sigma\nabla v\cdot\nabla {w} \rmd x + \sum_{m=1}^M\frac{1}{z_m}\int_{E_m}(v-V_m)({w}-{W_m})\rmd S
\end{equation} 
where the boundary restrictions of the appearing functions are identified with their corresponding traces. Thanks to the zero-mean condition on the second component of $ \Hbb^1 $, $ B $ is bounded and coercive \cite{Somersalo92} with respect to the norm $ \|\cdot\| $ induced by the natural scalar product 
\begin{equation}\label{eq:natural}
\big( \Vcal,\Wcal\big)_{H^1(\Omega)\times \Rbb^M} = \int_\Omega (\nabla v \cdot \nabla w + vw )\rmd x + V\cdot {W}.
\end{equation}
Consequently we have the following.
\begin{lemma}\label{lem:Lax-Milgram}
For an arbitrary $\phi \in (\Hbb^1)'$ there is a unique $\Ucal = \Ucal(\phi)\in\Hbb^1$ solving
\begin{equation}\label{eq:weak-mem}
B(\Ucal,\Wcal) = \phi(\Wcal), \qquad \forall \, \Wcal \in \Hbb^1
\end{equation} Moreover, the solution satisfies
\begin{equation}\label{eq:u-unibound}
\|\Ucal\| \leq C \max(\sigma^{-1}_-,z_+) \|\phi\|_{(\Hbb^1)'},
\end{equation}
where the constant $C > 0$ depends only on the geometry.
\end{lemma} 
\begin{remark}
Throughout the article, $C> 0$ is denotes generic constants that are independent of $z$ and that may change from one occasion to the next.
\end{remark}

Let us next return to the boundary value problem \eqref{eq:mem}. In order to have a weak version of the co-normal derivative in hand, we set 
\begin{equation}\label{eq:HD}
H_\sigma^s(\Omega):=\{v\in H^s(\Omega) : {\rm div}(\sigma\nabla v)\in L^2(\Omega)\}, \qquad s\in \Rbb,
\end{equation} 
and equip $H_\sigma^s(\Omega)$ with the graph norm 
\[ 
\|\cdot\|_{H^s_\sigma(\Omega)} := \|\cdot\|_{H^s(\Omega)}^2+\|{\rm div}(\sigma\nabla \cdot)\|_{L^2(\Omega)}^2. 
\]
The weak co-normal derivative is defined by $\gamma_1(\sigma) \colon H_\sigma^1(\Omega) \to H^{-1/2}(\partial\Omega) $ through
\begin{equation}
\label{eq:gamma}
\langle \gamma_1(\sigma)v, {w}\rangle = \int_{\Omega}\sigma\nabla v \cdot\nabla {w}\rmd x + \int_\Omega {\rm div}(\sigma\nabla v) {w} \rmd x.
\end{equation}
In the above formula $w\in H^{1/2}(\partial\Omega)$ is identified with its arbitrary $H^1(\Omega)$-extension because the right hand side is defined only up to addition of an $\smash{H^1_0(\Omega)}$-function to $w$. Indeed, this follows by density since the weak (distributional) definition of the differential operator is 
\[
\langle{\rm div}(\sigma\nabla v), {\varphi}\rangle = -\int_\Omega \sigma\nabla v \cdot\nabla{\varphi} \rmd x
\]
for $\varphi \in \mathscr{C}_0^\infty(\Omega)$ and $H_0^1(\Omega) = \overbar{\mathscr{C}_0^\infty(\Omega)}^{H^1(\Omega)}$. Using Green's formula, it is easy to check that $\gamma_1(\sigma)$ coincides with the standard conormal derivative for smooth enough functions and boundaries.

We conclude this section by observing the connection between \eqref{eq:mem} and \eqref{eq:weak-mem}.
\begin{theorem} \label{thm:weak_strong_sm}
Let $\Ucal = \Ucal(\phi_I)\in\Hbb^1$ be the solution of (\ref{eq:weak-mem}) with the second-member
 $\phi_I(\Wcal) := I\cdot{W}$. Then \ $\Ucal = (u,U)$ satisfies \eqref{eq:mem} with the co-normal derivative replaced with $\gamma_1(\sigma)u$. The converse statement also holds.
\end{theorem}
\begin{proof}
The essentials of the proof can be found in \cite{Somersalo92}.
\end{proof}

Next we focus on the behaviour of $\Ucal$ when the contact impedances tend to zero. A natural candidate
for the limit is the solution of the SM, which roughly correspond to 
the CEM problem \eqref{eq:mem} with vanishing contact impedances.

\section{Convergence to the shunt model}\label{sec:Gal}

\subsection{Shunt model}

In EIT, the SM models the idealistic case of perfect conduction between the body and the electrodes. Mathematically this corresponds to replacing the Robin condition in \eqref{eq:mem} by the Dirichlet condition corresponding to $z_m = 0$. In fact, this modification
causes a drop of a half Sobolev smoothness index in the solution (see next
section for the details). As a consequence, the SM has a slightly more complicated definition than the CEM; in particular, the last equation of \eqref{eq:mem} does not hold anymore as a standard integral.
For this reason, we focus on the variational formulation of the SM, equivalent to the standard formulation, but easier to handle and sufficient for our purposes.

For any closed subspace $\Vbb\subset \Hbb^1$, the Lax--Milgram lemma (cf. Lemma \ref{lem:Lax-Milgram}) guarantees the existence of a unique element $ \Ucal_{\Vbb} = \Ucal_\Vbb(\phi) \in \Vbb $ satisfying
\begin{equation}\label{eq:Galerkin}
B(\Ucal_{\Vbb},\Wcal) = \phi(\Wcal) \qquad \forall \ \Wcal\in\Vbb,
\end{equation}
for $\phi\in \Vbb'$. Note that $\Ucal_\Vbb$ can be interpreted as the $B$-orthogonal projection of $\Ucal$ onto the closed subspace $\Vbb$. It turns out that the solution of the SM problem is precisely the solution of \eqref{eq:Galerkin}, denoted by $\Ucal_0(\phi_I)$, with 
\begin{equation}\label{eq:H0}
\Vbb = \Hbb_0^1 := \{\Wcal\in\Hbb^1 \ : \ w|_{E_m} = W_m, \ m = 1,2,\ldots,M\}
\end{equation}
and $\phi = \phi_I$ from Theorem \ref{thm:weak_strong_sm}. Without going into details, we emphasize that this can, in essence, be shown by following the same lines of reasoning as in the proof of Theorem \ref{thm:weak_strong_sm}. However, we remind the reader that the lack of regularity
of $\Ucal_0(\phi_I)$ gives rise to the need for understanding the net current boundary condition in a weaker sense (see \eqref{eq:I} for the case of smooth geometry). Notice also that $\Ucal_0$ is independent of $z$ although $B$ is not.

\subsection{Convergence result}

For the rest of this section, we are interested in the existence and characterization of the limit of the solution corresponding to the CEM when
the contact impedances $\left\lbrace z_m\right\rbrace_{m=1}^M$ tend to zero on all electrodes.
We actually prove a slightly more general result:
For any given closed subspace $\smash{\Vbb \subset \Hbb^1}$,  $\Ucal_\Vbb $ always converges as $ z $ tends to zero, and its limit is $\Ucal_{\Vbb_0}\in \Vbb_0$, where 
\begin{equation}\label{eq:V0}
\Vbb_0 := \Vbb \cap \Hbb^1_0 = \{\Wcal = (w,W)\in \Vbb \ \colon \ w|_{E_m} = W_m, \ m=1,2,\ldots,M\}.
\end{equation}

\begin{proposition}\label{prop:zGal}
Let $ \Vbb\subset \Hbb^1 $ be a closed subspace and $\phi\in\Vbb'$. With  $ \Ucal_{\Vbb} = \Ucal_\Vbb(\phi) \in \Vbb$ and $ \Ucal_{\Vbb_0} = \Ucal_{\Vbb_0}(\phi)\in \Vbb_0 $ being defined as in \eqref{eq:Galerkin}, we have
\begin{equation}\label{eq:zlimit}
 \lim_{z\to 0}\Ucal_{\Vbb} = \Ucal_{\Vbb_0} 
\end{equation}
in the space $\Hbb^1$.
\end{proposition}
\begin{proof}
For clarity, we abbreviate the solutions by $\Vcal(z) = \Ucal_{\Vbb}$ and the bilinear form by $B(z)$. We perform an indirect argument and assume that \eqref{eq:zlimit} is false, i.e., there exists a sequence sequence 
$$
\{\Vcal^{(j)}\}_{j=1}^\infty := \{\Vcal(z^{(j)})\}_{j=1}^\infty, \qquad \lim_{j\to\infty}\Vcal^{(j)} \neq \Ucal_{\Vbb_0}
$$ 
where $ z^{(j)} \to 0$ in $(0,\infty)^M$ as $j\to \infty$. According to the uniform (with respect to $z$) bound \eqref{eq:u-unibound}, the function $z \mapsto \|\Vcal(z)\|$ is bounded. Hence, by the Banach--Alaoglu theorem, we may assume that there exists a $\Vcal\in\Vbb$ such that
$$
\lim_{j\to\infty} \Vcal^{(j)} = \Vcal \qquad \mathrm{weakly \ \, in \ \,} \Hbb^1.
$$
First we show that actually $\Vcal = \Ucal_{\Vbb_0}$. 

Using $z_m \leq z_+$ for any $m=1,2,\ldots,M$ and \eqref{eq:sigma}, we obtain from \eqref{eq:Galerkin} that
\begin{equation}\label{eq:vz}
\sum_{m=1}^M\int_{E_m}|v(z)-V_m(z)|^2\rmd S \leq C z_+ \|\phi\|_{\Vbb'} \to 0 \quad {\rm as} \ \ z_+\to 0.
\end{equation}
Hence the (weak) continuity of the trace operator yields
\[
\sum_{m=1}^M\int_{E_m}|v-V_m|^2\rmd S = \lim_{j\to\infty}\sum_{m=1}^M\int_{E_m}|v^{(j)}-V_m^{(j)}||v-V_m|\rmd S = 0.
\] 
Subsequently, Cauchy--Schwarz inequality and \eqref{eq:vz} give $v = V_m$ on every $E_m$, that is, $\Vcal \in \Vbb_0$. Similarly, by weak convergence, an arbitrary $\Wcal\in\Vbb_0$ satisfies 
\[
\int_\Omega \sigma\nabla v\cdot \nabla {w}\rmd x = \lim_{j\to\infty} \int_\Omega \sigma\nabla v^{(j)}\cdot\nabla {w}\rmd x = \lim_{j\to\infty}B(z^{(j)})(\Vcal^{(j)},\Wcal) = \phi(\Wcal),
\]
where the middle equality is a consequence of the vanishing boundary term. Thus uniqueness guarantees the claimed $\Vcal = \Ucal_{\Vbb_0}$. 

We are ready to derive the contradiction. By coercivity we estimate 
\begin{align}\label{eq:phivj}
\|\Vcal^{(j)}-\Ucal_{\Vbb_0}\|^2 &\leq C |B(z^{(j)})(\Vcal^{(j)}-\Ucal_{\Vbb_0},\Vcal^{(j)}-\Ucal_{\Vbb_0})|\nonumber \\[4pt] 
& = C |\phi(\Vcal^{(j)}-\Ucal_{\Vbb_0})|
\end{align}
with a constant $C > 0$ independent of $ j $. Note that the equality in \eqref{eq:phivj} follows by symmetricity of $\sigma$ and vanishing of suitable boundary terms. Therefore, by weak convergence the right-hand side of \eqref{eq:phivj} converges to zero as $ j\to\infty $. This implies strong convergence for the sequence $ \{\Vcal^{(j)}\}_{j=1}^\infty $ which contradicts the counter-assumption.
\end{proof}

There are certain special cases of Proposition \ref{prop:zGal} that are of particular interest. First of all, in the case $\Vbb = \Hbb^1$ it follows that  that the CEM solution converges to that of the SM with an unspecified rate as $ z \to 0$ in $(0,\infty)^M$. Secondly, any Galerkin approximation of the CEM converges to that of the SM problem as the contact impedances tend to zero.

\section{Regularity and convergence rates in the smooth setting}\label{sec:Reg}

\subsection{Regularity of the spatial part}

In this section we move onwards to study the Sobolev smoothness of the spatial potentials of CEM and SM in the case when all the predetermined attributes are smooth enough. We show in both cases that the $ H^1 $-regularity of the spatial potential is not the optimal. To avoid extra technicality, we assume that $\partial\Omega$, $\partial E_m$, $m=1,2,\ldots,M$ are all in the $\mathscr{C}^\infty$-class. We also suppose that in addition to satisfying \eqref{eq:sigma}, the conductivity $ \sigma$ belongs to $ \mathscr{C}^\infty(\overbar{\Omega};\Rbb^{n\times n}) $. 

At this point, we need to enlarge the domain and range of the co-normal derivative $ \gamma_1(\sigma)u $ in smooth domains. In what follows, we use the generic $ \langle \, \cdot \, ,\, \cdot \, \rangle $ to denote the dual evaluation between any pair $ H^{-s}(\partial\Omega) $ and $ H^s(\partial\Omega) $, $ s\geq 0 $; if there is danger of confusion, we give further specifications. By density (see \cite[\S 2 Theorem 7.3]{Lions72}), the operator $ \gamma_1(\sigma): \mathscr{C}^\infty(\overbar{\Omega}) \to \mathscr{C}^\infty(\partial\Omega) $, $ \varphi \mapsto \nu\cdot\sigma\nabla \varphi|_{\partial\Omega} $ extends to a bounded operator 
\begin{equation}\label{eq:extg}
\gamma_1(\sigma)\colon H_\sigma^s(\Omega) \to H^{s-3/2}(\partial\Omega)
\end{equation}
for any $ 0<s<2 $. 

Next we introduce a reference problem which will be used to infer the extra regularity of $ u,u_0 \in H^1(\Omega) $. Suppose $\Gamma\subset \partial\Omega$ is simply connected, non-empty and that $\partial\Gamma$ is of class $\mathscr{C}^\infty$. Then, it is well known that for any pair of data $ g\in H^{1/2}(\Gamma) $, $ f\in L^2(\Omega) $ and parameter $ \beta \geq 0 $, there exists a unique $ v_\beta \in H^1(\Omega) $ satisfying weakly
\begin{equation}
\label{eq:mixed}
\displaystyle{
{\rm div}(\sigma\nabla v_\beta) = f \quad {\rm in}\;\, \Omega, \quad  
\gamma_1(\sigma) v_\beta = 0 \quad {\rm on}\;\,\partial\Omega\setminus{\Gamma}, \quad  
v_\beta +\beta{\gamma_1(\sigma) v_\beta} = g  \quad {\rm on}\;\, \Gamma
}
\end{equation}
with the case $\beta = 0$ in the rightmost constraint interpreted as a Dirichlet condition. Using the properties of the Dirichlet-to-Neumann map and interpolation of Sobolev spaces, we obtain the following regularity estimate:

\begin{theorem}\label{thm:costabel} The solution $ v_\beta\in H^1(\Omega) $ of \eqref{eq:mixed} satisfies for all $ s \in (-\tfrac{1}{2},\tfrac{1}{2}) $
\begin{equation}\label{eq:beta}
\|v_\beta\|_{H^{1+s}(\Omega)}\leq C \big(\|g\|_{H^{1/2+s}(\Gamma)}+\|f\|_{L^2(\Omega)}\big) 
\end{equation}
with a constant $C>0$ independent of $\beta$.
\end{theorem}

\begin{proof} For simplicity, we assume that $f = 0$ as the below reasoning can be adjusted to the general case with a few simple modifications. The idea is essentially based on the proof of \cite[Corollary 4.4]{Costabel96}. In the case $\beta = 0$, the unique solution of \eqref{eq:mixed} exists and satisfies \cite{Savare97,Costabel96} an estimate
\begin{equation}\label{eq:v0}
\|v_0\|_{H^{1+s}(\Omega)} \leq C \|g\|_{H^{1/2+s}(\Gamma)}
\end{equation}
for all $s\in (-\tfrac{1}{2},\tfrac{1}{2})$. We define the associated (partial) Dirichlet-to-Neumann map by
\begin{equation}\label{eq:Lambda_G}
\Lambda = \Lambda_\Gamma(\sigma) \colon H^{s+1/2}(\Gamma)\to (H^{1/2-s}(\Gamma))', \quad g \mapsto (\gamma_1(\sigma)v_0)|_\Gamma,
\end{equation}
where 
\begin{equation}\label{eq:restr}
(\gamma_1(\sigma)v_0)|_\Gamma : w \mapsto \langle \gamma_1(\sigma)v_0, \widetilde{w}\rangle, \qquad \widetilde{w}\in H^{1/2-s}(\partial\Omega), \ \ \widetilde{w}|_\Gamma = w
\end{equation}
and $v_0$ is the solution of \eqref{eq:mixed} corresponding to $g$ and $f = 0$. Note that as a consequence of the zero-Neumann condition, \eqref{eq:restr} is independent of the choice of the extension $\widetilde{w}$. By the standard characterization of $H^{1/2-s}(\Gamma)$ by restrictions (see \cite[\S 1 Theorem 9.2]{Lions72}), \eqref{eq:extg} and \eqref{eq:v0}, we see that $\Lambda$ is bounded.

Suppose for now that $\beta > 0$ and $v_\beta$ solves \eqref{eq:mixed} for some $g\in L^2(\Gamma)$ and let ${\rm Id} \colon L^2(\Gamma) \to L^2(\Gamma)$ denote the identity map. As a consequence of the fact that $v_\beta$ trivially solves the corresponding mixed Dirichlet/Neumann problem, we can write 
\begin{equation}\label{eq:opeq}
({\rm Id} + \beta\Lambda)(v_\beta|_\Gamma) = g.
\end{equation}
The rest of the proof is analogous to that of \cite[Corollary 4.4]{Costabel96}; it relies on studying continuity properties of the operator\footnote{For the existence of the inverse operator between these spaces, we refer to \cite{Savare97,Costabel96}. }
\[
({\rm Id} + \beta\Lambda)^{-1} \colon (H^{1/2-s}(\Gamma))' \to H^{1/2+s}(\Gamma), \qquad \beta > 0
\] 
for all $s\in(-\tfrac{1}{2},\tfrac{1}{2})$ via interpolation of Sobolev spaces \cite{Lions72} and utilization of the continuity of \eqref{eq:Lambda_G}.
\end{proof}

\begin{remark}
In fact, the interpolation argument of the proof of \cite[Corollary 4.4]{Costabel96} yields also a convergence rate 
$$
\|v_\beta -v_0\|_{H^{1+s}(\Omega)} = O(\beta^{t-s})
$$ 
for any $\smash{t\in [s,\tfrac{1}{2})}$. However, this result is not directly applicable to $\smash{\|u-u_0\|_{H^{1+s}(\Omega)}}$ because in the CEM the electrode boundary data depends on $z$. Moreover, note that we cannot expand $({\rm Id} + \beta \Lambda)^{-1}$ in terms of the Neumann series because powers $\Lambda^j$, $j\geq 2$, are not well defined (cf.~\cite[Theorem 3.1]{Costabel96}).
\end{remark}

The next theorem shows that \eqref{eq:beta} can nevertheless be used to get an analogous norm estimate for $ u,u_0 \in H^1(\Omega) $.

\begin{theorem}\label{thm:sm-smooth}
The functions $ u,u_0 \in H^1(\Omega) $, i.e. the spatial parts of CEM and SM solutions, satisfy
\begin{equation}\label{eq:u-reg}
\|u\|_{H^{1+s}(\Omega)}\leq C \max(\sigma_-^{-1},z_+)|I|, \quad \|u_0\|_{H^{1+s}(\Omega)} \leq C \sigma_-^{-1}|I|
\end{equation} 
for any given $s<\tfrac{1}{2}$.
\end{theorem} 

\begin{proof}
Here we consider only $ u $ ($ u_0 $ can be handled with straightforward modifications). The idea is to use a suitable partition of unity to get local estimates from Theorem \ref{thm:costabel}. We claim that there exists a partition of unity $ \{\varphi_p\}_{p=1}^M \subset \mathscr{C}^2({\Omega}) $ satisfying 
\begin{equation}
\label{eq:1} \sum_{p=1}^M \varphi_p = 1 \ \ {\rm in} \ \ \Omega, \quad \varphi_p|_{E_m} = \delta_{pm}, \quad \nu\cdot\sigma\nabla \varphi_p = 0 \ \ {\rm on} \ \ \partial\Omega.
\end{equation}
 This set functions can be constructed in the following way. First use e.g.~the converse of trace theorem \cite{Necas12} to select functions $ \widehat{\varphi}_p \in \mathscr{C}^2({\Omega}) $ such that $ \widehat{\varphi}_p $ satisfies the latter two conditions of \eqref{eq:1}. Then defining the functions $ \varphi_p \in \mathscr{C}^2({\Omega}) $ by
\[
\varphi_1 := 1 - \sum_{p=2}^M \widehat{\varphi}_p, \quad \varphi_p = \widehat{\varphi}_p, \ \ p = 2,3,\ldots,M,
\] 
gives a set of functions satisfying also the remaining summability condition (recall that $\overbar{E}_m \cap \overbar{E}_p = \varnothing$ if $m\neq p$). We find out that $u_p := u\varphi_p$ satisfies a boundary value problem of form \eqref{eq:mixed}. Obviously we have 
\begin{equation}\label{eq:pdeum}
{\rm div}(\sigma\nabla u_p) = 2\sigma\nabla u \cdot \nabla\varphi_p + u{\rm div}(\sigma\nabla\varphi_p) \in L^2(\Omega)
\end{equation}
in the weak sense, implying the norm estimate $\|{\rm div}(\sigma\nabla u_p)\|_{L^2(\Omega)}\leq C\|u\|_{H^1(\Omega)}$. In addition, a straightforward calculation 
reveals
\[
\langle \gamma_1(\sigma) u_p, {g} \rangle = \langle \gamma_1(\sigma) u, {g}{\varphi}_p \rangle + \int_{\partial\Omega} (\nu\cdot\sigma\nabla \varphi_p){g}{u} \rmd S
\]
for any $g\in H^{1/2}(\partial\Omega)$. Therefore, by \eqref{eq:1} and \eqref{eq:weak-mem} we have
\begin{equation}\label{eq:um}
\gamma_1(\sigma) u_p = -\frac{1}{z_p}(u_p-U_p)\chi_p,
\end{equation}
where $\chi_p$ is the indicator function of $E_p$.
Consequently, Theorem \ref{eq:mixed} can be applied to the boundary value problem defined by \eqref{eq:pdeum} and \eqref{eq:um}. Hence, for any given $s<\tfrac{1}{2}$, the solution $u_p$ satisfies the norm estimate
\begin{align*}
\qquad \|u_p\|_{H^{1+s}(\Omega)} & \leq C \big (\|U_p\chi_p\|_{H^{1/2+s}(E_p)} + \|{\rm div}(\sigma\nabla u_p)\|_{L^2(\Omega)}\big) \\[6pt] 
& \leq C \big( |U| + \|u\|_{H^1(\Omega)}\big )\\[6pt]
& \leq C \max(\sigma_-^{-1},z_+)|I|
\end{align*}
where the middle estimate follows from \eqref{eq:pdeum}. Eventually, using the summability condition of \eqref{eq:1} and triangle inequality, the proof is concluded.
\end{proof}

Thanks to Theorem \ref{thm:sm-smooth} and the regularity of Neumann problem, an extra $\tfrac{1}{2}$ degree of smoothness can be obtained for the CEM solution.

\begin{corollary}\label{cor:cem-smooth}
For an arbitrary $s\in[0,\tfrac{1}{2})$ the function $ u $ belongs to $ H^{3/2+s}(\Omega) $ and moreover  
\begin{equation}\label{eq:cem-smooth}
\|u\|_{H^{3/2+s}(\Omega)} = O( z_-^{-s-\epsilon}) \quad {\rm as} \ \ z_- \to 0
\end{equation}
for any $\epsilon\in(0,1-s)$. Note that the upper bound goes to infinity
as $z_-$ goes to zero.
\end{corollary}
\begin{proof} 
We abbreviate 
\[ 
g_m := \frac{1}{z_m}(U_m-u)|_{E_m}\in H^{1/2}(E_m), \qquad m=1,2,\ldots,M
\] 
and let $ \widetilde{g}_m $ denote the extension of $ g_m $ to $ \partial\Omega $ by zero; the continuity of the extension \cite[\S 1 Theorem 7.4]{Lions72} implies
\[
\gamma_1(\sigma) u = \sum_{m=1}^M \widetilde{g}_m \in H^{1/2-\epsilon}(\partial\Omega), \quad \|\gamma_1(\sigma) u\|_{H^{1/2-\epsilon}(\partial\Omega)} \leq C \sum_{m=1}^M \|g_m\|_{H^{1/2-\epsilon}(E_m)}.
\]
By regularity of the Neumann problem \cite[\S 2 Remark 7.2]{Lions72} and interpolation between spaces $H^t(E_m),(H^{t}(E_m))'$ with $t \geq 0$ \cite[\S 1 Theorem 12.5]{Lions72}, we obtain
\[
\begin{split}
\inf_{c\in\Rbb}\|u+c\|_{H^{3/2+s}(\Omega)} \leq C \sum_{m=1}^M\|g_m\|_{H^{s}(E_m)} \leq C \sum_{m=1}^M \|g_m\|_{H^{1-\epsilon}(E_m)}^\theta \|g_m\|_{(H^{\epsilon}(E_m))'}^{1-\theta}
\end{split}
\]
where $ \theta = s +\epsilon $. Since $g_m\in L^2(E_m)$ is identified \cite{Lions72} with $ w\mapsto \int_{E_m}g_m w \rmd S $ in $(H^{\epsilon}(E_m))' $, we have 
$$
 \|g_m\|_{(H^{\epsilon}(E_m))'} \leq C \|\widetilde{g}_m\|_{{H}^{-\epsilon}(\partial\Omega)}. 
$$
Consequently, the fact that the $E_m$ do not overlap each other, the continuity of $\gamma_1(\sigma)$ from $H^{3/2-\epsilon}_\sigma(\Omega)$ to $H^{-\epsilon}(\partial \Omega)$ (see \eqref{eq:extg}), and \eqref{eq:u-reg} together yield
\begin{equation}\label{eq:posteriori}
\inf_{c\in\Rbb}\|u+c\|_{H^{3/2+s}(\Omega)} \leq C |I|^{1-\theta}\sum_{m=1}^M z_m^{-\theta}\|u-U_m\|_{H^{1-\epsilon}(E_m)}^\theta
\end{equation}
with $ \theta = s+ \epsilon $ and $\epsilon\in[0,\tfrac{1}{2})$. By triangle inequality, trace theorem \cite{Lions72} and the fact that $U_m\in\Rbb$ is a constant, we may further estimate the right hand side of \eqref{eq:posteriori} using
\[
\|u-U_m\|_{H^{1-\epsilon}(E_m)} \leq C \big (\|u\|_{H^{3/2-\epsilon}(\Omega)} + |U_m|\big ).
\] 
Thus, by Theorem \ref{thm:sm-smooth} and \eqref{eq:u-unibound}, we get
\begin{equation}\label{eq:infest}
\inf_{c\in \Rbb}\|u+c\|_{H^{3/2+s}(\Omega)} \leq C |I|\sum_{m=1}^M z_m^{-s-\epsilon}.
\end{equation}

In order to manipulate the quotient norm in \eqref{eq:infest}, we recall that by basic properties of Sobolev inner product \cite{Wloka82}, there holds
\[
(w,1)_{H^{3/2+s}(\Omega)} = (w,1)_{L^2(\Omega)} = \int_\Omega w \rmd x
\]
for all $w\in H^{3/2+s}(\Omega)$. Thus we have
\[
\begin{split}
\inf_{c\in\Rbb}\|u+c\|_{H^{3/2+s}(\Omega)}^2 & = \inf_{c\in\Rbb}\left\|u-\frac{1}{|\Omega|}\int_\Omega u\rmd x+c\right\|_{H^{3/2+s}(\Omega)}^2 \\ 
& = \inf_{c\in\Rbb}\left\{ \left\|u-\frac{1}{|\Omega|}\int_\Omega u\rmd x\right\|_{H^{3/2+s}(\Omega)}^2 + c^2|\Omega| \right\} \\ 
& = \left\|u-\frac{1}{|\Omega|}\int_\Omega u\rmd x\right\|_{H^{3/2+s}(\Omega)}^2 \\ 
& = \|u\|_{H^{3/2+s}(\Omega)}^2-\frac{1}{|\Omega|}\left|\int_\Omega u \rmd x\right|^2 \\[8pt]
& \geq \|u\|_{H^{3/2+s}(\Omega)}^2 - \|u\|_{L^2(\Omega)}^2
\end{split}
\]
where the last estimate is a direct consequence of Cauchy--Schwarz inequality. Applying the above estimate, \eqref{eq:infest} and \eqref{eq:u-unibound}, we obtain
\[
\|u\|_{H^{3/2+s}(\Omega)} \leq C \bigg(\|u\|_{L^2(\Omega)} + |I|\sum_{m=1}^M z_m^{-s-\epsilon}\bigg) \leq C |I|\bigg(\max(\sigma_-^{-1},z_+) + \sum_{m=1}^M z_m^{-s-\epsilon}\bigg)
\]
which implies the claim.
\end{proof}

Identifying $\gamma_1(\sigma)u$ with the boundary current density it follows \cite[\S 1 Theorem 9.8]{Lions72} that  
\[
\gamma_1(\sigma)u|_{E_m} \in H^{1+s}(E_m) \subset \mathscr{C}^0(E_m), \qquad s\in(\tfrac{1}{2},1), \ \ \Omega \subset \Rbb^n, \ \ n=2,3.
\]
Therefore, in particular, we have $\gamma_1(\sigma)u \in L^\infty(\partial\Omega)$. This is not true for $\gamma_1(\sigma)u_0$ since it even falls outside of $L^2(\partial\Omega)$.
In the special case $\sigma \equiv 1$, $n=2$ for \eqref{eq:mixed}, the drop in Sobolev regularity was characterized in \cite{Costabel96} (see also \cite{Pidcock95}) by classifying the type of the singularities of $v_\beta$ at the transition points of boundary conditions. Using a singular decomposition technique, in the case $\beta > 0$, it was demonstrated that the most severe singularity is of type $r\log r$ whereas in the case $\beta = 0$ it is $\smash{r^{1/2}}$ with $(r,\theta)$ denoting the polar coordinates centered at the transition point in question. 

\subsection{Convergence rates}

In order to take advantage of the regularity provided by \eqref{eq:u-reg} in deriving convergence rates, we need the following lemma related to the approximation of trivially extended Sobolev functions by bump functions:
\begin{lemma}\label{lem:bump}
Let $ \Gamma\subset \partial\Omega $ be a connected set with a $ \mathscr{C}^\infty $-boundary. Suppose that $ g\in H^s(\Gamma) $ for some $ s\in[0,1/2) $ and denote by $ \widetilde{g} \in L^2(\partial\Omega) $ the extension of $ g $ to $ \partial\Omega $ by zero. Then $ \widetilde{g}\in H^s(\partial\Omega) $ and there exists a sequence of $ \mathscr{C}^\infty(\partial\Omega) $-functions supported in $ \Gamma $ that converges to $ \widetilde{g} $ in $ H^s(\partial\Omega) $.
\end{lemma}
\begin{proof}
By the density of compactly supported functions for any Sobolev exponent $ s\in[0,1/2] $ \cite[\S 1 Theorem 11.1]{Lions72}, it is possible to fix a sequence of functions $ (\varphi_j)_{j=1}^\infty \subset \mathscr{C}_0^\infty(\Gamma)$ which converges to $ g $ in the norm of $ H^s(\Gamma) $. As the zero extension $ \widetilde{\varphi}_j $ remains in $\mathscr{C}^\infty(\partial\Omega) $, the continuity of the zero extension operator for $ s\in[0,1/2) $ \cite[\S 1 Theorem 7.4]{Lions72} implies that $ \widetilde{g}\in H^s(\partial\Omega) $ and that the smooth functions $ \widetilde{\varphi}_j $ converge to $ \widetilde{g} $ in $ H^s(\partial\Omega) $.    
\end{proof}

Considering the smoothness given by Theorem \ref{thm:sm-smooth}, the net current condition for the SM can be interpreted in the following way.

\begin{proposition}\label{pr:u0-uz}
In $\mathscr{C}^\infty$-regular geometry $ \Ucal_0 \in \Hbb^1$ satisfies
\begin{equation}\label{eq:I}
\langle \gamma_1(\sigma) u_0 , \chi_m\rangle = I_m, \qquad m=1,2,\ldots,M
\end{equation}
where the dual evaluation can be taken between $H^{-s}(\partial\Omega)$ and $H^s(\partial\Omega)$ for arbitrary $s\in (0,\tfrac{1}{2})$ and $\chi_m$ is the indicator function of $E_m$.
\end{proposition}
\begin{proof}
Let $g\in \mathscr{C}^\infty(\partial\Omega)$ and $ s\in (0,\tfrac{1}{2}) $ be arbitrary. According to Lemma \ref{lem:bump} we can pick a sequence $ (\varphi_j)_{j=1}^\infty \subset \mathscr{C}^\infty_0(\partial\Omega\setminus E) $ such that 
\begin{equation}\label{eq:phichi} 
\lim_{j\to\infty}\varphi_j = \chi_{\partial\Omega\setminus E} \quad {\rm in} \quad H^s(\partial\Omega). 
\end{equation} 
By basic properties of Sobolev norm and \eqref{eq:phichi} we have
\[
\lim_{j\to\infty}\|g\varphi_j - g\chi_{\partial\Omega\setminus E}\|_{H^s(\partial\Omega)} \leq \|g\|_{\mathscr{C}^1(\partial\Omega)}\lim_{j\to\infty}\|\varphi_j-\chi_{\partial\Omega\setminus E}\|_{H^s(\partial\Omega)} = 0
\] 
and hence by continuity  
\[
\langle\gamma_1(\sigma) u_0,{g}\chi_{\partial\Omega\setminus E}\rangle = \lim_{j\to\infty}\langle\gamma_1(\sigma) u_0,{g}{\varphi}_j\rangle = 0
\]
where the last equality is a consequence of the variational problem in $ \Hbb_0^1 $ defining $ \Ucal_0 $ (cf. \eqref{eq:Galerkin}) and the fact $ {\rm supp}\, g\varphi_j \subset \partial\Omega\setminus E $.
Therefore, it holds
\begin{equation}\label{eq:chi}
\langle\gamma_1(\sigma) u_0,{g}\rangle = \langle\gamma_1(\sigma) u_0,{g}\chi_E\rangle 
\end{equation}
for any $ g\in \mathscr{C}^\infty(\partial\Omega) $. Choosing suitable test functions $ g $ that are constants on the electrodes, and recalling \eqref{eq:Galerkin} and that the electrodes do not overlap, we arrive at the alleged result. 
\end{proof}

Equation \eqref{eq:I} allows us to estimate $\|\Ucal-\Ucal_0\|$ by using the coercivity of $B$ to obtain the following:

\begin{theorem}\label{thm:zconv}
The discrepancy between the CEM and SM solutions satisfies
\begin{equation}\label{eq:zconv}
\|\Ucal -\Ucal_0\| \leq C |I|z_+^s
\end{equation}
for any $ s \in [0,\tfrac{1}{2}) $, with $C = C(\Omega,E,\sigma) > 0$.
\end{theorem}

\begin{proof} As a consequence of \eqref{eq:I} and \eqref{eq:chi} we write 
\begin{align}\label{eq:BW}
B(\Ucal-\Ucal_0,\Wcal) & = -\int_\Omega \sigma\nabla u_0 \cdot \nabla{w}\rmd x +I\cdot{W}\nonumber\\[4pt]
& = -\sum_{m=1}^M\langle \gamma_1(\sigma)u_0, {w}\chi_m \rangle +\sum_{m=1}^M \langle \gamma_1(\sigma) u_0,{W_m}\chi_m \rangle \nonumber \\
& = -\sum_{m=1}^M \langle \gamma_1(\sigma) u_0, ({w}-{W_m})\chi_m \rangle  
\end{align}
for all $ \Wcal\in\Hbb^1 $ where the middle equality follows from the definition \eqref{eq:gamma} of the conormal derivative. The choice $ \Wcal = \Ucal-\Ucal_0 $ further leads to 
\begin{equation}\label{eq:Bz}
B(\Ucal-\Ucal_0,\Ucal-\Ucal_0) = -\sum_{m=1}^M \langle \gamma_1(\sigma) u_0, ({u}-{U_m})\chi_m \rangle.
\end{equation}
Taking the coercivity of $ B $ into account, it is sufficient to obtain a bound of the desired form for the right hand side of \eqref{eq:Bz}. By the continuity of $ \gamma_1(\sigma) $ we can estimate
\begin{align}\label{eq:z-7}
|\langle \gamma_1(\sigma) u_0, ({u}-{U_m})\chi_m \rangle|  & \leq C \|\gamma_1(\sigma)u_0\|_{H^{t-1/2}(\partial\Omega)} \|(u-U_{m})\chi_m\|_{H^{1/2-t}(\partial\Omega)} \nonumber\\[6pt]
& \leq C \|u_0\|_{H^{1+t}(\Omega)} \|u-U_{m}\|_{H^{1/2-t}(E_m)}
\end{align}
for an arbitrary $t\in (0,\tfrac{1}{2})$. Applying the Robin boundary condition suitably (see proof of Corollary \ref{cor:cem-smooth}), we may use the continuity of $\gamma_1(\sigma)$ to deduce
\[
\|u - U_m\|_{(H^\epsilon(E_m))'} = \|z_m\gamma_1(\sigma)u\|_{H^{-\epsilon}(\partial\Omega)}\leq C z_m \|u\|_{H^{3/2-\epsilon}(\Omega)}
\]
for any $\epsilon \in (0,t+\tfrac{1}{2}]$. Therefore, by interpolation and trace theorem, we get
\begin{align}\label{eq:z-8}
\|u-U_{m}\|_{H^{1/2-t}(E_m)}\nonumber & \leq C \|u-U_{m}\|_{(H^{\epsilon}(E_m))'}^{1-\theta}\|u-U_{m}\|_{H^{1-\epsilon}(E_m)}^\theta \nonumber \\[6pt]
&\leq C z_m^{1-\theta} \|u-U_{m}\|_{H^{3/2-\epsilon}(\Omega)}^\theta\|u\|_{H^{3/2-\epsilon}(\Omega)}^{1-\theta}
\end{align}
with $ \theta = \tfrac{1}{2} +\epsilon -t $. Finally, applying \eqref{eq:u-reg} and taking the square root, we obtain \eqref{eq:zconv} with $ s = \tfrac{1-\theta}{2} = \tfrac{1+2(t-\epsilon)}{4} $.
\end{proof}

Before concluding the section, we point out that convergence rates can be obtained also in other norms. In particular, the next corollary reveals that the electrode voltages $U\in\Rbb_\diamond^M$ converge twice as fast as the potential inside the body.

\begin{corollary}\label{cor:adjoint-z} For the solutions $\Ucal,\Ucal_0\in\Hbb^1$ there holds
\begin{equation}\label{eq:adjoint-z}
\|u-u_0\|_{L^2(\Omega)}+|U-U_0| \leq C |I|z_+^{2s}
\end{equation}
for any $ s\in [0,\tfrac{1}{2}) $. Furthermore, the spatial components satisfy
\begin{equation}\label{eq:zspat}
\|u-u_0\|_{H^{1+s}(\Omega)} \leq C z_+^{1/2-s-\epsilon}
\end{equation}
for any $s\in[0,\tfrac{1}{2})$ and $\epsilon\in (0,1/2-s)$. 
\end{corollary}
\begin{proof} 
The first part is proved using a standard ``duality argument'' \cite{Brenner08}. As both error terms in \eqref{eq:adjoint-z} can be handled separately but analogously, it is sufficient to consider the term $|U-U_0|$. Define $ \Vcal \in \Hbb^1 $ as the unique solution to the problem
\begin{equation}\label{eq:adjoint}
B(\Vcal,\Wcal) = J \cdot W \qquad \forall \ \Wcal \in \Hbb^1
\end{equation}
where $J\in \Rbb^M$. In particular, by \eqref{eq:adjoint} and symmetricity we get
\begin{align*}
J \cdot (U-U_0) = B(\Vcal,\Ucal-\Ucal_0) = B( \Ucal-\Ucal_0,\Vcal-\Vcal_0).
\end{align*}
Consequently, expressing the Euclidean norm via supremum 
yields
\begin{equation}\label{eq:squarenorm}
|U-U_0| = \max_{|J| = 1} |B( \Ucal-\Ucal_0,\Vcal-\Vcal_0)|. 
\end{equation}
The idea is to derive a bound for the right-hand quantity without resorting to continuity of $B$. Instead, we will apply Sobolev regularity and interpolation to get 
\begin{equation}\label{eq:idea}
|B( \Ucal-\Ucal_0,\Vcal-\Vcal_0)| \leq C |I||J|z_+^{2s}
\end{equation}
for all $s\in[0,\tfrac{1}{2})$. Combining this with \eqref{eq:squarenorm} then gives \eqref{eq:adjoint-z}.

Let us demonstrate how to obtain \eqref{eq:idea}. Since by \eqref{eq:BW} the modulus of the rightmost expression of \eqref{eq:squarenorm} is bounded by 
\begin{equation}\label{eq:ast}
\sum_{m=1}^M |\langle \gamma_1(\sigma) u_0 ,(v-V_{m},\chi_m \rangle|,
\end{equation}
it is sufficient to find a suitable bound for this quantity. 
According to \eqref{eq:z-7} and \eqref{eq:z-8}, we deduce
\begin{align}\label{eq:2s}
|\langle \gamma_1(\sigma) u_0,(v-V_{m})\chi_m \rangle| &\leq C \|u_0\|_{H^{1+t}(\Omega)}\|v-V_m\|_{H^{1/2-t}(E_m)} \nonumber \\[4pt] & \leq C z_m^{1-\theta}\|u_0\|_{H^{1+t}(\Omega)}\|v-V_m\|_{H^{3/2-\epsilon}(\Omega)}^\theta \|v\|_{H^{3/2-\epsilon}(\Omega)}^{1-\theta} \nonumber \\[4pt]
& \leq C z_m^{1-\theta} |I| |J|
\end{align}
for $t,\epsilon,\theta$ as in the proof of Theorem \ref{thm:zconv}. Combining \eqref{eq:squarenorm}, \eqref{eq:idea} and \eqref{eq:2s} we get
\[
|U-U_0| \leq C |I| z_+^{2s}
\]
where $s = \tfrac{1+2(t-\epsilon)}{4}$ can be chosen freely from the interval $[0,\tfrac{1}{2})$. 

The second part of the claim is again an application of interpolation. Utilizing the partition of unity \eqref{eq:1} with Theorem \ref{thm:costabel} ($\beta = 0$) we get
\[
\begin{split}
\|u-u_0\|_{H^{1+s}(\Omega)} &\leq C \bigg( \|u-u_0\|_{H^1(\Omega)} + \sum_{m=1}^M \|u-u_0\|_{H^{1/2+s}(E_m)} \bigg )\\
& \leq C \bigg( \|\Ucal-\Ucal_0\| + \sum_{m=1}^M\|u-U_m\|_{H^{1/2+s}(E_m)} \bigg)
\end{split}
\]
for any $s\in [0,\tfrac{1}{2})$; note that the bottom estimate is merely based on trivial estimation and the electrode boundary condition of $u_0$. By interpolation (cf. \eqref{eq:z-8}) we further estimate
\[
\|u-U_m\|_{H^{1/2+s}(E_m)} \leq C z_m^{1-\theta}|I|
\] 
with $\theta = \tfrac{1}{2}+s+\epsilon$ and any $\epsilon\in (0,1/2-s)$. The claim is a direct consequence of this since by \eqref{eq:zconv} we can estimate $ \|\Ucal-\Ucal_0\| $ as required.
\end{proof}

To conclude the section, we remark that it is not a difficult task to generalize the above results to the case where $z_m \to 0$ possibly only for $m$ in a subset of $\{1,2,\ldots,M\}$. 

\section{Numerical tests}\label{sec:Num}

We proceed with two numerical examples related to some of the results presented in Sec.~\ref{sec:Gal} and \ref{sec:Reg}. In Sec.~\ref{subsec:Con} we test whether a convergence rate indicated by \eqref{eq:zconv} is apparent in the corresponding FE approximations by piecewise linears. We are also interested whether the numerical electrode voltages converge noticeably faster than the spatial potentials (cf.~\eqref{eq:adjoint-z}). The other example is presented in Sec.~\ref{subsec:FE}. There we numerically study what kind of an effect different values of $ z $ have on the convergence (with respect to the maximal triangle diameter $ h $) of the FE approximation of the CEM. Although this question is of practical interest on its own, it can also be understood as an indirect numerical verification of the observed regularity drop (see \eqref{eq:u-reg}, \eqref{eq:cem-smooth}). 

\subsection{Convergence test}\label{subsec:Con}

In the first numerical example, the test object $ \Omega $ is a regular hexadecagon with all of the 16 corners lying on the unit sphere. The conductivity is constant $ \sigma \equiv 1 $ and there are $M=8$ identical, equidistant electrodes each of which covers exactly one boundary edge of $\Omega$. We compute approximate solutions using the FE method with piecewise linear basis functions. Denoting the space defined by the triangle-wise linear polynomials by $ P_1 $, we select 
$$ 
\Vbb = P_1\oplus \Rbb_\diamond^M; 
$$ 
this Galerkin space is used to compute $ \Ucal_\Vbb $ i.e.~a FE approximation of the CEM. The corresponding $\Vbb_0$ is defined as in \eqref{eq:V0} and we use it to approximate the SM. For a detailed description of the assembly and computation of the system matrices, we refer to \cite{Vauhkonen99}.

\begin{figure}[h!]
\begin{center}
\includegraphics[width=0.45\textwidth]{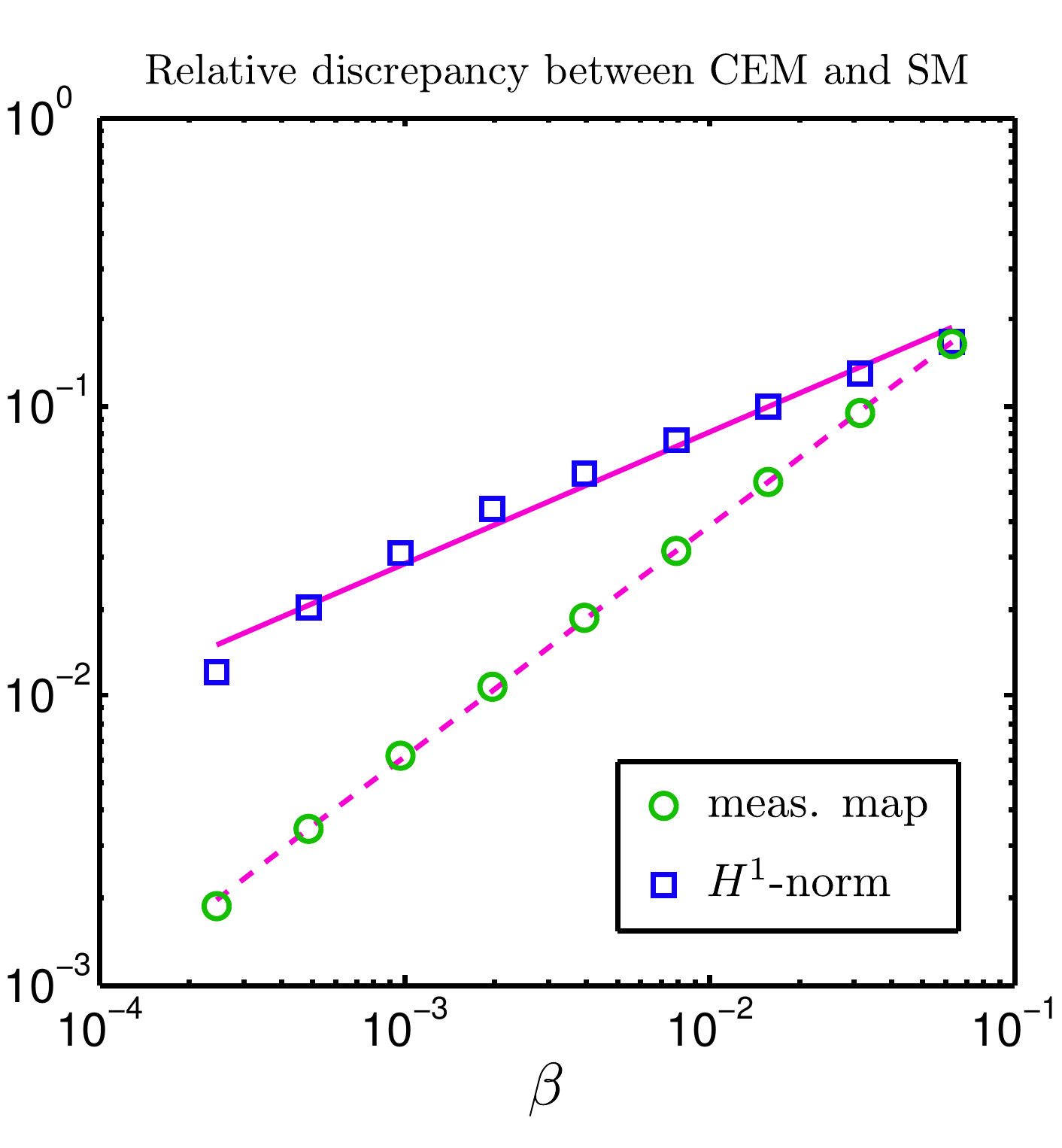}
\raisebox{62pt}{\begin{tabular}{|c|c|}
\multicolumn{2}{c}{} \\ 
\hline
$  {\color{white} \displaystyle{\sum}}\|u_\Vbb-u_{\Vbb_0}\|_{H^1(\Omega)} $ & $O(\beta^{0.4567})$ \\
 \hline
 ${\color{white} \displaystyle{\sum}} \|\Ucal_\Vbb-\Ucal_{\Vbb_0}\| \ \ \ \ \ \, \, \,$ & $ O(\beta^{0.4582}) $\\
 \hline
 ${\color{white} \displaystyle{\sum}} \|u_\Vbb-u_{\Vbb_0}\|_{L^2(\Omega)} $ & $O(\beta^{0.7163})$ \\
\hline
 \qquad ${\color{white} \displaystyle{\sum}} \|R_\Vbb-R_{\Vbb_0}\|_{M\times M} \ \ \  \,\,\, $ & $ O(\beta^{0.8011}) $\\
 \hline
\end{tabular}}
\caption{\label{fig:z} On the left: discrepancy as a function of $z\equiv \beta = {\rm constant}$. On the right: convergence rates in different norms. The error in the measurement map is calculated in the operator norm of $\Rbb^{M\times M}$. All the computations were performed using a fixed triangulation such that away from the boundary the mesh parameter was $ h= 0.079 $ and near the boundary $ h = 0.005 $.}
\end{center}
\end{figure}

In Fig.~\ref{fig:z} the discrepancies between the CEM and SM are visualized as the constant contact impedance 
$$
z = [\beta,\beta,\ldots,\beta]^{\rm T} \to 0 
$$ 
The examined solutions $ \Ucal_\Vbb = \Ucal_\Vbb^{(2)}$ and $ \Ucal_{\Vbb_0} = \Ucal_{\Vbb_0}^{(2)} $ are computed using the input current 
$$
I^{(2)} = [\cos(2\pi m/M)]_{m=1}^M.
$$ 
Similar rates were also obtained for other input currents. We have also considered the measurement matrix 
$$
R_\Vbb\in \Rbb^{M\times M} \quad ( R_{\Vbb_0} \ \, \mathrm{respectively})
$$
defined as the unique matrix having the following two properties:  it maps every $I\in\Rbb_\diamond^M$ to $\smash{U_\Vbb\in\Rbb^M_\diamond}$, where $\smash{\Ucal_\Vbb = (u_\Vbb,U_\Vbb)\in\Hbb^1}$ is the corresponding solution to \eqref{eq:Galerkin}, and its null space is spanned by $\smash{[1,1,\ldots,1]^{\rm T}\in\Rbb^M}$. Note that the $\smash{I^{(2)}}$ defined above is an eigenvector of $R_\Vbb$ ($R_{\Vbb_0}$ respectively) corresponding to the second smallest eigenvalue (see e.g.~\cite{Somersalo92}). 

First of all, we observe that the convergence indicated by Proposition \ref{prop:zGal} appears to take place. The estimated convergence rates in the tabular of Fig. \ref{fig:z} are obtained by a least squares fit of linear functions in $\log \beta$. Although the results fall below the rates predicted by Theorem \ref{thm:zconv} and Corollary \ref{cor:adjoint-z}, it is reasonable to hypothesize that for any $s\in [0,\tfrac{1}{2})$ there exists a fine enough triangulation of $\Omega$ such that an (infinitely precise) numerical scheme will detect the rates
\[ 
\begin{array}{lll}
& \|u-u_0\|_{H^1(\Omega)}  = O(\beta^s), \qquad & \|\Ucal-\Ucal_0\| = O(\beta^s), \\[4pt]
& \|u-u_0\|_{L^2(\Omega)} = O(\beta^{2s}), \qquad & \|R-R_0\|_{M\times M} = O(\beta^{2s}).
\end{array}
\] 
We further observe that qualitatively the obtained estimates are fairly well in accordance with the theory in the sense that $ \|\Ucal_\Vbb - \Ucal_{\Vbb_0}\| \approx O(\beta^{0.4582}) $ is far from linear whereas the error $ \smash{\|R_\Vbb - R_{\Vbb_0}\|_{M\times M} \approx O(\beta^{0.8011}) }$ decays roughly twice as fast in the limit $ \beta \to 0 $.

\subsection{The effect on the convergence of FE approximation}\label{subsec:FE}

We continue working in the same geometry as in the previous example. However, in this case we do not fix the triangulation of $ \Omega $ but instead use a set of gradually sharpening uniform triangulations\footnote{The authors admit that this is not reasonable in practical applications. Due to the high regularity of $ u $ away from $ \partial\Omega $ it is advisable to use adaptive meshing (cf. e.g.~\cite{Hakula12}).} to estimate the convergence rate of the FE approximation by $ \displaystyle{\Vbb = P_1\oplus \Rbb_\diamond^M} $. More precisely, for each member of a set of constant contact impedances $ z \equiv \beta > 0$, we compute for $ \Ucal_\Vbb$ an estimated convergence rate with respect to decreasing mesh parameter $ 0<h\to 0+ $. 

In order to derive a priori error estimates with respect to $ h $, we note that for any given function $ v\in H^{3/2+s}(\Omega) $, $s \in (0,\tfrac{1}{2})$, it can be shown using suitable polynomial interpolator(s) \cite{Brenner08,Stenberg08}, that 
\begin{equation}\label{eq:intp-error}
\inf_{w\in P_1}\|v-w\|_{H^1(\Omega)} \leq C h^{1/2+s}\|v\|_{H^{3/2+s}(\Omega)},
\end{equation}
where the constant $C > 0$ depends on $s$. This and the hypothesis that $ u $ satisfies \eqref{eq:cem-smooth}\footnote{It is well known that in polygonal domains this is not the case e.g.~if the boundary has concave angles. For a detailed discussion on the topic, see for example \cite{Costabel96} and the references therein.} lead to convergence rates with respect to $h$. 
Namely, by C\'ea's lemma and the fact that $ \Vbb = P_1\oplus \Rbb_\diamond^M $, we have
\begin{equation}\label{eq:h1}
\|\Ucal-\Ucal_\Vbb\| \leq C \inf_{\Wcal \in \Vbb}\|\Ucal-\Wcal\| \leq C \inf_{w\in P_1}\|u-w\|_{H^1(\Omega)} \leq C |I|h^{1/2+s}
\end{equation}
with the rightmost constant being of order $ C =  O(\beta^{-1/2-s-\epsilon}) $ for any $ \epsilon \in (0,1-s) $. Moreover, applying a ``dual technique'' as in the proof of \eqref{eq:adjoint-z}, we obtain that
\begin{equation}\label{eq:h2}
\|u-u_\Vbb\|_{L^2(\Omega)}+|U-U_\Vbb| \leq C |I|h^{1+2s}
\end{equation}
with $ C = O(\beta^{-3/2-s-\epsilon}) $. 
Because of the constants' explosion in the limit $ \beta \to 0+ $, one may anticipate that (when using uniform triangulations) the computational detection of rates corresponding to $ s$ close to $1/2 $ becomes increasingly demanding. 

\begin{figure}[h!]
\begin{center}
\includegraphics[width=0.49\textwidth]{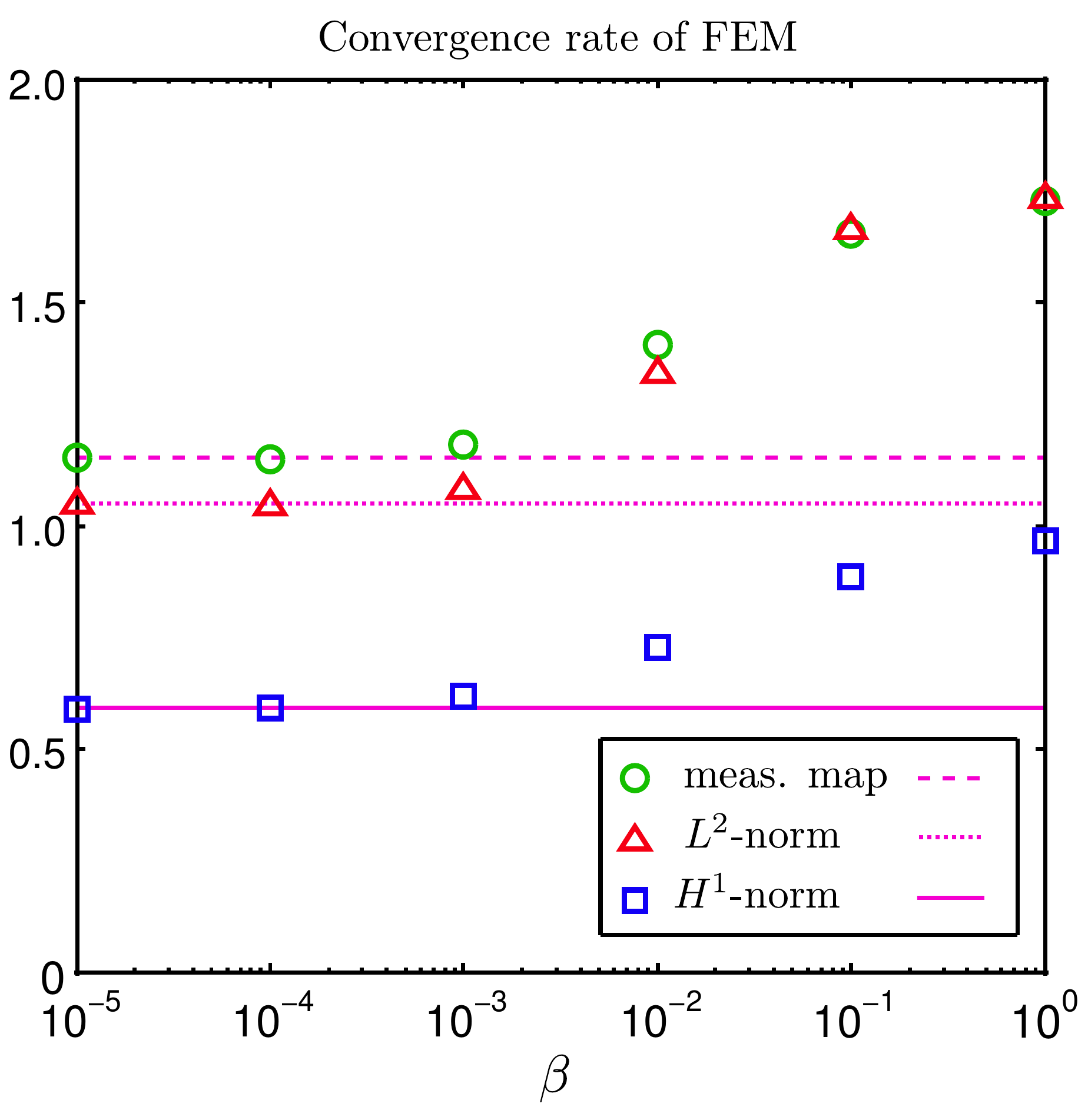} 
\caption{\label{fig:h} Convergence rate of the FE approximation by piecewise linears as a function of $\beta$. On the vertical axis is the estimated (by least squares) slope in $ \log h $. The $L^2$ and $H^1$-errors are computed over $\Omega$ for the interior potential, and the error in the measurement map is measured in the operator norm of $\Rbb^{M\times M}$. The horizontal lines illustrate the respective estimates obtained for the SM (dashed, meas. map; dotted, $L^2(\Omega)$-norm; solid, $H^1(\Omega)$-norm).}
\end{center}
\end{figure}

\begin{remark}
(a) Let $\kappa_h$ denote the condition number of the matrix 
corresponding to the FE discretization of the bilinear form \eqref{eq:sesq}.
A simple computation shows that 
$$
\kappa_h \geq \frac{C(\Omega,h)}{\sigma_+ \beta}
$$
where $C(\Omega,h) > 0$ is a constant independent of $\beta$. Therefore, inversion of the linear system --- and hence the computation of $\Ucal_\Vbb$ --- is ill-conditioned for $\beta$ close to zero. 

(b) 
Implementation of gradient based EIT reconstruction algorithms usually require a numerical approximation of $U'$, i.e., the Fr\'echet derivative of $U$ with respect to a finite dimensional $\sigma$ \cite{Darde13,Heikkinen02}. Since $U'$ depends also on $u$, it is reasonable to use a finer triangulation in the approximation of $\smash{U'}$ than in the simulation of the electrode data $U$ (see \eqref{eq:h1} and \eqref{eq:h2}).
\end{remark}

The exact solution is approximated here by taking $ \hat{\Vbb} $ with a mesh parameter $ \hat{h} $ considerably smaller than those of any of the explored $\Vbb$. In Fig.~\ref{fig:h} the estimated $h$-convergence rates are illustrated in different norms as a function of $\beta$. The applied current inputs are chosen as above in Sec.~\ref{sec:Con}. Again, each one of the estimated rates is obtained from a least squares fit of a linear function in $\log h$. For comparison, the calculation is performed also for the SM case i.e. using $ \Vbb_0 $ and $ \hat{\Vbb}_0 $ as Galerkin spaces, respectively.

\section{Conclusions}\label{sec:Con}

We have demonstrated that the CEM converges to the SM as the contact impedance $z$ tends to zero. The same was also shown for their FE approximations. In smooth domains, we proved that the $H^1$-discrepancy between the CEM and the SM is of the order $O(z^s)$, $0\leq s < \tfrac{1}{2}$. Using a duality argument, it was possible to demonstrate that (in theory) the difference between the corresponding electrode measurement maps is almost linear $O(z^s)$, $0\leq s < 1$. The first numerical experiment verified these rates to a certain extent. We also pointed out that the spatial part of the SM solution has Sobolev regularity of a half degree less than that of the CEM. 
The results of the latter numerical experiment support this drop in regularity, and point out that the FE method gives a more accurate approximation for the CEM when $z \gg 0$. 

\section{Acknowledgements}\label{sec:Acknow}

The authors would like to thank professor Nuutti Hyv\"onen for carefully reading the manuscript and suggesting improvements.

\end{document}